\documentclass[reqno, 10pt]{amsart}
\usepackage{amssymb}
\usepackage{amsmath}
\usepackage[mathscr]{euscript}
\usepackage{hyperref}
\usepackage{graphicx}

\usepackage{amssymb}
\usepackage{hyperref}
\RequirePackage[dvipsnames]{xcolor} 
\definecolor{halfgray}{gray}{0.55} 
\definecolor{webgreen}{rgb}{0,0.5,0}
\definecolor{webbrown}{rgb}{.6,0,0} \hypersetup{%
  colorlinks=true, linktocpage=true, pdfstartpage=3,
  pdfstartview=FitV,%
  breaklinks=true, pdfpagemode=UseNone, pageanchor=true,
  pdfpagemode=UseOutlines,%
  plainpages=false, bookmarksnumbered, bookmarksopen=true,
  bookmarksopenlevel=1,%
  hypertexnames=true,
  pdfhighlight=/O,
  urlcolor=webbrown, linkcolor=RoyalBlue,
  citecolor=webgreen, 
  pdftitle={},%
  pdfauthor={},%
  pdfsubject={2000 MAthematical Subject Classification: Primary:},%
  pdfkeywords={},%
  pdfcreator={pdfLaTeX},%
  pdfproducer={LaTeX with hyperref}%
}

\usepackage{enumitem}

\newtheorem{theorem}{Theorem}[section]
\newtheorem*{theorem*}{Theorem}
\newtheorem{lemma}[theorem]{Lemma}
\newtheorem{corollary}[theorem]{Corollary}
\newtheorem{proposition}[theorem]{Proposition}

\newtheorem{remark}[theorem]{Remark}

\newtheorem{conjecture}{Conjecture}[section]

\newtheorem{ltheorem}{Theorem}

\def\R{\mathbb{R}}
\def\real{\mathbb{R}}

\def\proj{\mathbb{P}^1}

\def\id{\operatorname{id}}

\def\cW{\mathcal{W}}

\def\quand{\quad\text{and}\quad}

\def\hA{\hat{A}}
\def\hrho{\hat{\rho}}

\def\SL{SL(2,\real)}

\def\Pone{{\mathbb{P}^{1}}}

\def\hA{\hat{A}}

\def\hm{\hat{m}}

\def\hnu{\hat{\nu}}

\def\Pone{\mathbb{P}^1}

\newcommand{\norm}[1]{{\left\lVert  #1  \right\rVert}}
\newcommand{\abs}[1]{{\left\lvert  #1  \right\rvert}}

\begin{document}

\title
[The set of cocycles with nonvanishing Lyapunov exponents is open]
{The set of fiber-bunched cocycles with nonvanishing Lyapunov exponents over a partially hyperbolic map is open}

\author{Lucas Backes}

\address{Departamento de Matem\'atica, Universidade Federal do Rio Grande do Sul, Av. Bento Gon\c{c}alves 9500, CEP 91509-900, Porto Alegre, RS, Brazil.}
\email{lhbackes@impa.br }

\author{Mauricio Poletti}

\address{LAGA -- Universit\'e Paris 13, 99 Av. Jean-Baptiste Cl\'ement, 93430 Villetaneus, France.}
\email{mpoletti@impa.br}

\author{Adriana S\'anchez}

\address{IMPA -- Estrada D. Castorina 110, Jardim Bot\^anico, 22460-320 Rio de Janeiro, Brazil.}
\email{asanchez@impa.br}

\date{\today}

\keywords{Lyapunov exponents, Partially hyperbolic systems, Linear cocyles}
\subjclass[2010]{Primary: 37H15, 37A20; Secondary: 37D25}
  
  \begin{abstract}
We prove that the set of fiber-bunched $SL(2,\mathbb{R})$-valued H\"{o}lder cocycles with nonvanishing Lyapunov exponents over a volume preserving, accessible and center-bunched partially hyperbolic diffeomorphism is open. Moreover, we present an example showing that this is no longer true if we do not assume accessibility in the base dynamics.
\end{abstract}

\maketitle

\section{Introduction}
Given an invertible measure preserving transformation $f\colon (M, \mu)\to (M,\mu)$ of a standard probability space and a measurable function $A\colon M\to GL(d, \mathbb{R})$ we define the \emph{linear cocycle} over $f$ by the dynamically defined products
\begin{equation}\label{def:cocycles}
A^n(x)=
\left\{
	\begin{array}{ll}
		A(f^{n-1}(x))\ldots A(f(x))A(x)  & \mbox{if } n>0 \\
		Id & \mbox{if } n=0 \\
		(A^{-n}(f^{n}(x)))^{-1}=A(f^{n}(x))^{-1}\ldots A(f^{-1}(x))^{-1}& \mbox{if } n<0. \\
	\end{array}
\right.
\end{equation}
The simplest examples of linear cocycles are given by derivative transformations of smooth dynamical systems: the cocycle \emph{generated} by $A(x)=Df(x)$ over $f$ is called the \emph{derivative cocycle}. Taking as an example the \emph{hyperbolic theory} of Dynamical Systems where one can understand certain dynamical properties of $f$ by studying the action of $Df$ on the tangent space, one can hope that by studying properties of linear cocycles one can also deduce some properties of $f$. Nevertheless, the notion of linear cocycle is much more general and flexible, and arises naturally in many other situations as in the spectral theory of Schr\"{o}dinger operators, for instance.

In this short note we are interested in the asymptotic behavior of $A^n(x)$. More precisely, we are interested in understanding certain regularity properties of \emph{Lyapunov exponents}. These objects measure the asymptotic rates of contractions and expansions along different directions and are one of the most fundamental notions in dynamical systems.

It is well known that, in general, Lyapunov exponents can be very sensitive as functions of the cocycle. For instance, Bochi \cite{Boc-un, Boc02} proved that in the space of $SL(2,\mathbb{R})$-valued continuous cocycles over an aperiodic map, if a cocycle is not hyperbolic, then it can be approximated by cocycles with zero Lyapunov exponents. In particular, there are cocycles with positive Lyapunov exponents that are accumulated by cocycles with zero Lyapunov exponents. Moreover, Bocker and Viana \cite{BockerV} constructed an example over a hyperbolic map showing that the same phenomenon can happen in the H\"{o}lder realm. Furthermore, when the base dynamic is far from being hyperbolic, for example, when $f$ is a rotation on the circle, Wang and You \cite{WaY15}, showed that having non-zero Lyapunov exponents is not an open property even in the $C^{\infty}$ topology.

In order to construct their example, Bocker and Viana exploited the fact that the cocycle is not \emph{fiber-bunched}. In fact, it was shown by Backes, Butler and Brown \cite{BBB} that in the fiber-bunched setting over a hyperbolic map the Lyapunov exponents vary continuously with respect to the cocycle and, in particular, cocycles with positive Lyapunov exponents can not be approximate by cocycles with zero Lyapunov exponents.

In the present work we are interested in understanding the case when the cocycle still have some regularity properties, namely, it is fiber-bunched but the base dynamics exhibit some mixed behaviour of hyperbolicity and non-hyperbolicity, that is, the map $f$ is \emph{partially hyperbolic}. In fact, we show that if $f$ is chaotic enough and $A$ is fiber-bunched then the Bochi phenomenon can not occur. More precisely, (see Section \ref{sec: statements} for detailed definitions),

\begin{theorem}\label{t: introduction}
If $(f,\mu)$ is a volume preserving partially hyperbolic accessible and center-bunched diffeomorphism  and $A:M\to SL(2,\mathbb{R})$ is a H\"{o}lder continuous fiber-bunched map with nonvanishing Lyapunov exponents, then $A$ can not be accumulated by cocycles with zero Lyapunov exponents.
\end{theorem}

Moreover, we show that the accessibilty assumption in the previous result is necessary. More precisely, 

\begin{theorem}
There exists a volume preserving partially hyperbolic and center-bunched diffeomorphism $f$ and a H\"{o}lder continuous fiber-bunched map $A$ with non-zero Lyapunov exponents which is approximated by cocycles with zero Lyapunov exponents.
\end{theorem}

\section{Statements} \label{sec: statements}
Let $f:M\to M$ be a $C^{r}$, $r\geq 2$, diffeomorphism defined on a compact manifold $M$, $\mu$ an ergodic $f$-invariant Borel probability measure and let $A:M\to SL(2,\R)$ be an $\alpha$-H\"{o}lder continuous map. This means that there exists a constant $C>0$ such that
\begin{displaymath}
\norm{A(x)-A(y)} \leq C d(x,y)^{\alpha}
\end{displaymath}
for all $x,y\in M$ where $\norm{A}$ denotes the operator norm of a matrix $A$, that is, $\norm{A} =\sup \lbrace \norm{Av}/\norm{v};\; \norm{v}\neq 0 \rbrace$. Let $H^{\alpha}(M)$ denote the space of all such $\alpha$-H\"{o}lder continuous maps. We endow this space with the $\alpha$-H\"{o}lder topology which is generated by the norm 
$$\norm{A}_{\alpha}=\sup _{x\in M}\norm{A(x)}+\sup_{x\neq y}\frac{\norm{A(x)-A(y)}}{d(x,y)^{\alpha}}.$$

\subsection{Lyapunov exponents}

It follows from the subadditive ergodic theorem of Kingman \cite{Ki68} that there exists a full $\mu$-measure set $\mathcal{R}^{\mu}\subset M$, whose points are called \emph{$\mu$-regular points}, such that for every $x\in \mathcal{R}^{\mu}$ the limits
\begin{displaymath}
\lambda^u(A,x)=\lim _{n\to  \infty} \frac{1}{n} \log \norm{A^n(x)} \text{ and }
\lambda^s(A,x)=\lim _{n\to  \infty} \frac{1}{n} \log \norm{(A^n(x))^{-1}}^{-1}
\end{displaymath}
exist. We call such limits \emph{Lyapunov exponents}. Moreover, when $\lambda^u(A,x)\neq \lambda^s(A,x)$ it follows from a famous theorem of Oseledets \cite{Ose68}  that there exists a decomposition $\mathbb{R}^2=E^{u,A}_{x}\oplus E^{s,A}_{x}$, called the \emph{Oseledets decomposition}, into vector subspaces depending measurably on $x$ such that for every $x\in \mathcal{R}^{\mu}$,
\begin{equation}\label{eq: Lyap and Osel}
A(x)E^{*,A}_{x}=E^{*,A}_{f(x)} \text{ and }  \lambda ^{*}(A,x) =\lim _{n\to \pm  \infty} \dfrac{1}{n}\log \| A^n(x)v\| 
\end{equation}
for every non-zero $v\in E^{*,A}_{x}$ and $*\in \{u,s\}$. Furthermore, since the Lyapunov exponents are $f$-invariant, ergodicity of $\mu$ implies that they are constant for every $x\in \mathcal{R}^{\mu}$. In this case we write $\lambda^u(A,x)=\lambda^u(A,\mu)$ and $\lambda^s(A,x)=\lambda^s(A,\mu)$.

\subsection{Partial Hyperbolicity}\label{subsec: PH}
A diffeomorphism $f:M \to M$ of a compact $C^r$ manifold $M$, $r\geq 1$, is said to be \emph{partially hyperbolic} if there exists a non-trivial splitting of the tangent bundle
\begin{equation*}
TM=E^s\oplus E^c\oplus E^u
\end{equation*}
invariant under the derivative $Df$, a Riemannian metric $\|\cdot\|$ on $M$, and positive
continuous functions $\nu$, $\hat{\nu}$, $\gamma$, $\hat{\gamma}$ with $\nu$, $\hat{\nu}<1$
and $\nu<\gamma<{\hat\gamma}^{-1}<{\hat\nu}^{-1}$ such that, for any unit vector $v\in T_xM$,
\begin{alignat*}{2}
& \|Df(x)v \| < \nu(x) & \quad & \text{if } v\in E^s(x),
 \\
\gamma(x) < & \|Df(x)v \|  < {\hat{\gamma}(x)}^{-1} & & \text{if } v\in E^c(x),
 \\
{\hat{\nu}(x)}^{-1} < & \|Df(x)v\| & &  \text{if } v\in E^u(x).
\end{alignat*}
All three sub-bundles $E^s$, $E^c$, $E^u$ are assumed to have positive dimension. We say that $f$ is \emph{center-bunched} if 
\begin{displaymath}
\nu < \gamma \hat{\gamma} \text{ and } \hat{\nu} < \gamma \hat{\gamma}.
\end{displaymath}
We need this hypothesis because we are going to use the results of \cite{ASV13}. From now on, we take $M$ to be endowed with the distance $d:M\times M\to \real$ associated to such a Riemannian structure.

Suppose that $f:M\to M$ is a partially hyperbolic diffeomorphism, then the stable and unstable bundles $E^s$ and $E^u$ are uniquely integrable and their integral manifolds form two transverse continuous foliations $\cW^s$ and $\cW^u$, whose leaves are immersed sub-manifolds of the same class of differentiability as $f$. These foliations are referred to as the \emph{strong-stable} and \emph{strong-unstable} foliations. They are invariant under $f$, in the sense that
$$
f(\cW^s(x))= \cW^s(f(x)) \quand f(\cW^u(x))= \cW^u(f(x)),
$$
where $\cW^s(x)$ and $\cW^u(x)$ denote the leaves of $\cW^s$ and $\cW^u$, respectively,
passing through any $x\in M$. We say that $f$ is \emph{accessible} if $M$ and $\emptyset$ are the only $su$-saturated sets. This means that, except of $\emptyset$, $M$ is the only set that is a union of entire strong-stable and strong-unstable leaves.

\subsection{Fiber-bunched cocycles} Let $f:M\to M$ be a $C^r$ partially hyperbolic map on a compact manifold $M$ and $A:M\to SL(2,\R)$ be an $\alpha$-H\"{o}lder continuous map. We say that the cocycle generated by $A$ over $f$ is \emph{fiber-bunched} if 
\begin{equation*}
\|A(x)\|\|A(x)^{-1}\|\nu (x)^\alpha <1 \text{ and } \|A(x)\|\|A(x)^{-1}\|\hnu (x)^\alpha <1
\end{equation*}
for every $x\in M$. As a shorthand for this notion, since our base dynamics $f$ is going to be fixed, we simply say that $A$ is fiber-bunched. Observe that this is an open condition in $H^\alpha(M)$.

\subsection{Main results} The main results of this note are the following. Recall that a measure $\mu$ is in the \emph{Lebesgue class} if it is generated by a volume form.

\begin{ltheorem}\label{teoA}
Let $f:M\to M$ be a $C^{r}$, $r\geq 2$, partially hyperbolic, volume preserving,
center-bunched and accessible diffeomorphism defined on a compact manifold $M$ and $\mu$ an ergodic $f$-invariant measure in the Lebesgue class. If $A\in H^{\alpha}(M)$ is fiber-bunched and $\lambda^u(A,\mu)>\lambda^s(A,\mu)$ then $A$ can not be accumulated by cocycles with zero Lyapunov exponents.  
\end{ltheorem}

We observe that a similar result can be stated in terms of $GL(2,\mathbb{R})$-valued cocycles changing `cocycles with zero Lyapunov exponents' by `cocycles with just one Lyapunov exponent'. Indeed, by continuity of $A$ and connectedness of $M$ (which follows from the accessibility), either $\det(A(x))>0$ for every $x\in M$ or $\det(A(x))<0$ for every $x\in M$. Suppose we are in the first case (the other case can be easily deduced from this one). Then, given $A:M \rightarrow GL(2,\mathbb{R})$ consider $g_{A}:M\rightarrow \mathbb{R}$ defined by $g_{A}(x)=(\det A(x))^{\frac{1}{2}}$ and $B:M \rightarrow SL(2,\mathbb{R})$ such that $A(x)=g_{A}(x)B(x)$. Therefore,
\begin{displaymath}
\lambda ^{u/s}(A,\mu)=\lambda ^{u/s}(B,\mu) + \int \log (g_{A}(x))\  d\mu(x),
\end{displaymath}
and consequently,
\begin{displaymath}
\lambda ^{u}(A,\mu)=\lambda ^{s}(A,\mu) \Longleftrightarrow \lambda ^{u}(B,\mu)=0=\lambda ^{s}(B,\mu).
\end{displaymath}

As already mentioned at the introduction, we also present an example showing that the accessibilty assumption in the previous theorem is necessary. More precisely,

\begin{ltheorem}\label{teoB} 
There exists a volume preserving partially hyperbolic and center-bunched diffeomorphism $f$ and a H\"{o}lder continuous fiber-bunched map $A$ with non-zero Lyapunov exponents which is approximated by cocycles with zero Lyapunov exponents.
\end{ltheorem}

In light of the previous results, we are lead to make the following conjecture which is in the same spirit as the conjectures proposed by Viana \cite{LLE} in the hyperbolic setting.

\begin{conjecture}\label{conjecture}
Under the assumptions of Theorem \ref{teoA} the Lyapunov exponents of H\"{o}lder continuous $\SL$-valued cocycles vary continuously in the set of fiber-bunched cocycles.
\end{conjecture}

As a consequence of \cite[Corollary 4]{LMY} (see also \cite{ASV13}) it follows that the previous conjecture is true in an open and dense subset of the fiber-bunched elements of H\"{o}lder continuous $\SL$-valued cocycles giving more evidences of its veracity.

\section{Preliminary results}
In this section we recall some classical notions and present some useful results that are going to be used in the proof of our main theorem. Let $f:M\to M$, $A\in H^\alpha(M)$ and $\mu$ be as in Theorem \ref{teoA}.

\subsection{Accessibility and holonomies}\label{sec:accessibility}
 
Given $x,y\in M$, we write $x\sim^s y$ whenever $y\in \cW^s(x)$. Observe that this is an equivalence relation and moreover, is $f$-invariant. That is, if $x\sim^s y$ then $f(x)\sim^s f(y)$.  Analogously, we write $x\sim^u z$ if $z\in \cW^u(x)$.
 
An \emph{$su$-path} from $x$ to $y$ is a path connecting $x$ and $y$ which is a concatenation of finitely many subpaths, each of which lies entirely in a single leaf of $\cW^s$ or a single leaf of $\cW^u$. Every sequence of points $x=z_0,z_1,\dots,z_n=y$, such that $z_i\sim^* z_{i+1}$ for $*=s$ or $u$, and $i=0,\dots,n-1$ defines a unique $su$-path. An \emph{su-loop} or a \emph{closed $su$-path} is an $su$-path beginning and ending at the same point. If $\gamma_1$ is an $su$-path given by $z_0,\dots,z_n$ and $\gamma_2$ is an $su$-path given by $z'_0,z'_1,\dots,z'_m$, with $z'_0=z_n$, we define $\gamma_1\wedge \gamma_2$ as the $su$-path given by $z_0,\dots,z_n,z'_1,\dots,z'_m $.

We say that an $su$-path $\gamma$ defined by the sequence $x=z_0,z_1,\dots,z_n=y$ is a \emph{$(K,L)$-path} if $n\leq K$ and $d_{\cW^*}(z_{i+1},z_i)\leq L$ for every $i=1,\ldots ,n-1$ where $d_{\cW^*}$ is the distance induced by the Riemannian strucutre on the submanifold $\cW^*$ for $*=s,u$. For simplicity we write $x\sim^*_L y$ if $d_{\cW^*}(z_{i+1},z_i)\leq L$ for every $i=1,\ldots ,n-1$. Observe that, by the compactness of $M$ and continuity of stable manifolds of bounded size, the space of $(K,L)$-paths is compact. In particular,
 
\begin{lemma}{\cite[Lemma~4.5]{Wliv}}\label{l.KLpath}
There exist constants $K$ and $L$ such that every pair of points in $M$ can be connected by an $(K,L)$-path.
\end{lemma}

For every pair of points $x,y\in M$ so that $x\sim^s y$, our fiber-bunched assumption assures that the limit 
 $$
 H^{s,A}_{xy}=\lim_{n\to+\infty}A^n(y)^{-1}\circ A^n(x)
 $$
exists (see \cite[Proposition~3.2]{ASV13}). Moreover, for every $L>0$,
 $$
(x,y,A)\to H^{s,A}_{xy} \text{ is continuous on } \cW^s_L \times H^\alpha(M) 
$$
where $\cW^s_L = \{ (x,y)\in M\times M; x\sim^s_L y \}$ (see \cite[Remark~3.4]{ASV13}). In particular,
\begin{remark}\label{r.unifcont} 
Given a sequence $\{A_k\}_{k\in \mathbb{N}}$ converging to $A$ in $H^\alpha(M)$, since $\cW^s_L$ is compact, 
$$
\{\cW^s_L \ni (x,y) \to H^{s,A_k}_{xy}\}_{k\in \mathbb{N}}
$$
is equi-continuous for $k$ sufficiently large.
\end{remark}

The family of maps $H^{s,A}_{xy}$ is called \emph{an stable holonomy} for the cocycle $(A, f)$. It is easy to verify that (see \cite[Proposition 3.2]{ASV13}) for $x\sim^sy$ and $z\sim^sy$,
\begin{displaymath}
H^{s,A}_{xx}=Id \quad \text{ and } \quad H^{s,A}_{xy}=H^{s,A}_{zy}\circ H^{s,A}_{xz} 
\end{displaymath}
and 
\begin{displaymath}
H^{s,A}_{f^j(x)f^j(y)}=A^j(y)H^{s,A}_{xy}A^j(x)^{-1} \quad \forall j\geq 0.
\end{displaymath}

Similarly, for $x\sim^u y$  we define the \emph{unstable holonomy} $H^{u,A}_{xy}$ as the stable holonomies for $(A^{-1},f^{-1})$. If $\gamma$ is the $su$-path defined by the sequence $z_0,z_1,\dots,z_n$ then we write $H^A_\gamma=H^{*,A}_{z_{n-1}z_n}\circ \ldots \circ H^{*,A}_{z_{0}z_1}$ for $*\in \{s,u\}$.

\subsection{Disintegrations and $su$-invariance}
We say that a measure $m$ on $M\times \proj$ \emph{projects} on $\mu$ if $\pi_{*}m=\mu$ where $\pi$ is the canonical projection $\pi: M\times \proj \to M$. Observe that any such measure admits a \emph{disintegration} with respect to the partition $\{\{x\}\times \proj\}_{x\in M}$ and the measure $\mu$, that is, there exists a family of measures $\{m_x\}_{x\in M}$ on $\{\{x\}\times \proj\}_{x\in M}$ so that for every measurable $B\subset M\times \proj$,
\begin{itemize}
\item $x\to m_x(B)$ is measurable,
\item $m_x(\{x\}\times \proj)=1$ and 
\item $m(B)=\int_M m_x (B\cap (\{x\}\times \proj) ) d\mu(x)$.
\end{itemize}
Moreover, such disintegrantion is essentially unique \cite{Rok62}. Identifying each fiber $\{x\}\times \proj$ with $\proj$, we can think of $x\to m_x$ as a map from $M$ to the space of probability measures on $\proj$ endowed with the weak$^*$ topology.

Let $F_A:M\times \proj \to M\times \proj$ be the map given by
$$F_A(x,v)=(f(x), [A(x)v])$$
and $m$ be an $F_A$-invariant measure projecting on $\mu$. We say that  $m$ is \emph{$s$-invariant} if there exists a total measure set $M^s\subset M$ such that for every $x,y\in M^s$ satisfying $x\sim^s y$ we have ${H^{s,A}_{xy}}_*m_x=m_y$. Such measure $m$ is also known as an \emph{$s$-state}. Analogously, we say that $m$ is \emph{$u$-invariant} (or an \emph{$u$-state}) if the same is true replacing stable by unstable in the previous definition. We say that $m$ is \emph{$su$-invariant} if it is simultaneously $s$-invariant and $u$-invariant. The main property of $su$-ivariant measures is the following

\begin{proposition}{\cite[Theorem D]{ASV13}}
Any $F_A$-invariant measure $m$ projecting on $\mu$ which is $su$-invariant admits a disintegration $\{m_x\}_{x\in M}$ for which $M^s=M^u=M$ and so that $m_x$ depends continuously on the base point $x\in M$ in the weak$^*$ topology. 
\end{proposition}

\subsection{Trivial holonomies on $su$-loops} \label{sec: tivial holonomies}
In this section we explain how in certain specific situations we can perform a change of coordinates that makes the cocycle $(A,f)$ constant without changing its Lyapunov exponents.

Let us assume that $H^A_{\gamma}=\id$ for every $su$-loop $\gamma$ with at most $3K$ legs and each of them with length at most $ L$. Recall that we call such loops $(3K, L)$-loops. In particular, $H^A_\gamma=\id$ for \emph{every} $su$-loop $\gamma$. Indeed, observe initially that if $\gamma$ is a $(2K, L)$-path from $x$ to $y$ then, by Lemma \ref{l.KLpath}, there exists a $(K,L)$-path $\gamma'$ from $x$ to $y$ so that $H^A_\gamma=H^A_{\gamma'}$. In fact, if $-\gamma'$ denotes the path $\gamma'$ with opposite orientation then $\gamma \wedge (-\gamma ')$ is a $(3K, L)$-loop and $$H^A_{\gamma}\circ (H^A_{\gamma '})^{-1}=H^A_{\gamma}\circ H^A_{-\gamma '} = H^A_{\gamma \wedge (-\gamma ')}=\id.$$ 
Hence, $H^A_\gamma=H^A_{\gamma'}$. Now, taking any $su$-loop $\gamma$ with an arbitrary number of legs whose lengths are at most $ L$ we can decompose it as $\gamma=\gamma_1\wedge \cdots\wedge \gamma_k$, where every $\gamma_i$ is a $(K, L)$-path. In particular, $\gamma_{k-1}\wedge\gamma_k$ is a $(2K, L)$-path and by the previous argumment we can replace it by a $(K,L)$-path $\gamma'_{k-1}$ with the same starting and ending points and, so that $H^A_{\gamma_{k-1}\wedge\gamma_k}=H^A_{\gamma'_{k-1}}$. Thus, taking $\gamma'=\gamma_1\wedge \cdots\wedge\gamma_{k-2}\wedge \gamma'_{k-1}$ we have that $\gamma$ and $\gamma'$ have the same starting and ending points and $H^A_\gamma=H^A_{\gamma'}$.  Repeating this procedure a finite number of times we get some $(K,L)$-loop $\gamma''$ such that $H^A_\gamma=H^A_{\gamma''}=\id$. Finally, observing that any $su$-loop $\gamma$ can be transformed into an $su$-loop with legs of size at most $L$ just by breaking one ``large'' leg into several with smaller sizes we conclude that $H^A_\gamma=\id$ for every $su$-loop proving our claim. As a consequence we get that if $\gamma$ is an $su$-path connecting $x$ and $y$ then $H^A_{\gamma}$ \emph{does not depend} on $\gamma$. In fact, if $\gamma _1$ and $\gamma _2$ are $su$-paths connecting $x$ and $y$ then $\gamma _1\wedge (-\gamma _2)$ is an $su$-loop and thus $H^A_{\gamma _1}\circ (H^A_{\gamma _2})^{-1} =H^A_{\gamma _1}\circ H^A_{-\gamma _2} = H^A_{\gamma _1 \wedge (-\gamma _2)}=\id$ as claimed. Let us denote this common value simply by $H^A_{xy}$. From the properties of the holonomies and the fact that any two points $x,y\in M$ can be connected by a $(K,L)$-path it follows that
\begin{itemize}
\item $H^A_{yz}H^A_{xy}=H^A_{xz}$,
\item $A(y)H^A_{xy}=H^A_{f(x)f(y)}A(x)$,
\item $A\to H^A_{xy}$ is uniformly continuous for any pair of points $x,y\in M$ and
\item $\|H^A_{xy}\|\leq N$ for some $N>0$ and any $x,y\in M$.
\end{itemize}

Fix $x\in M$ and, given $y\in M$, consider the following transformation
$$
\hA(y)=H^A_{f(y)x} A(y)H^A_{xy}.
$$
Then, $\hA^2(y)=\hA(f(y))\hA(y)=H^A_{f^2(y)x} A(f(y))H^A_{xf(y)}H^A_{f(y)x} A(y)H^A_{xy}$ and consequently $\hA^2(y)=H^A_{f^2(y)x}  A^2(y)H^A_{xy}$. More generally, $\hA^n(y)=H^A_{f^n(y)x}  A^n(y)H^A_{xy}$ for every $n\in \mathbb{N}$ and consequently $(\hA,f)$ and $(A,f)$ have the same Lyapunov exponents. Moreover, for any $z,y\in M$,
$$
\begin{aligned}
\hA(z)^{-1}\hA(y)=&\left(H^A_{f(z)x} A(z)H^A_{xz}\right)^{-1}H^A_{f(y)x} A(y)H^A_{xy}\\
=&H^A_{zx} A(z)^{-1}H^A_{xf(z)}H^A_{f(y)x} A(y)H^A_{xy}\\
=&H^A_{zx} A(z)^{-1}H^A_{f(y)f(z)} A(y)H^A_{xy}\\
=&H^A_{zx} A(z)^{-1} A(z)H^A_{yz}H_{xy}\\
=&H^A_{zx}H^A_{yz}H_{xy}\\
=&H^A_{zx}H^A_{xz}\\
=&\id.
\end{aligned}
$$
In particular, $\hA$ is constant and consequently its largest Lyapunov exponent is the logarithm of the norm of the greatest eigenvalue of $\hA$. Summarizing, if $H^A_{\gamma}=\id$ for every $(3K,L)$-loop $\gamma$ then we can perform a change of coordinates that makes the cocycle $(A,f)$ constant without changing its Lyapunov exponents. This is going to be used in Section \ref{sec: concluison}.

\subsection{$\SL$ matrices and invariant measures on $\proj$}\label{sec: SL2r} The following result plays an important part in our proof below.

\begin{proposition}\label{l.hyperbolic.matrix}
For each $n\in \mathbb{N}$, let $L_n$ be a $ \SL$ matrix so that $L_n\xrightarrow{n\to +\infty} \id$ and let $\eta_n$ be an $L_n$-invariant measure on $\proj$ so that $\eta_n\xrightarrow{n\to +\infty} \frac{1}{2}( \delta_p+\delta_q)$ for some $p,q\in \proj$ with $p\neq q$. Then for every $n$ sufficiently large either $L_n$ is hyperbolic or $L_n=\id$.
\end{proposition}
\begin{proof}
The proof is by contradiction. We start observing that as $L_n$ converges to the identity all the matrices have positive trace for $n$ sufficiently large. Consequently, if $L_n$ is not the identity we have three posibilities: if the trace $tr(L_n)>2$ then the matrix $L_n$ is hyperbolic, if $tr(L_n)<2$ then the matrix $L_n$ is elliptic and is conjugated to a rotation of angle $\theta_n=\arccos (\frac{tr(L_n)}{2})$ and if $tr(L_n)=2$ then the matrix $L_n$ is parabolic and is non diagonalizable with both eigenvalues equal to 1. 

Suppose initially that all the matrices $L_n$ have $tr(L_n)<2$. In particular, for each $n\in \mathbb{N}$ there exists $P_n\in \SL$ so that $L_n=P_n^{-1} R_{\theta_n} P_n$ where $R_{\theta_n}$ stands for the rotation of angle $\theta_n$. Moreover, since $tr(L_n)\xrightarrow{n\to +\infty} 2$, we get that $\theta_n\xrightarrow{n\to +\infty} 0$.

Now, for each $n\in \mathbb{N}$ let us consider $\nu_n={P_n}_*\eta_n$ which is an $R_{\theta_n}$-invariant measure. We start observing that there exists a subsequence $\{n_j\}_j$ so that $\nu_{n_j}\xrightarrow{j\to +\infty} \text{Leb}$ where $\text{Leb}$ stands for the Lebesgue measure on $\proj$. Indeed, if $\theta_n$ is an irrational number then we know that the only $R_{\theta_n}$-invariant measure is $\text{Leb}$. In particular, $\nu_n=\text{Leb}$. Thus, if there are infinitely many values of $n$ for which  $\theta_n$ is an irrational number we are done.

Suppose then that $\theta_n$ is a rational number for every $n\in \mathbb{N}$. In particular, $R_{\theta_n}$ is periodic and denoting by $q_n$ its period, since $\theta_n\xrightarrow{n\to +\infty} 0$, we have that $q_n\xrightarrow{n\to +\infty} + \infty$. 

In what follows we make an abuse of notation thinking of $\proj$ as $[0,1]$ identifying the extremes of the interval.

Let $\varphi:\proj \to \mathbb{R}$ be a continuous map and $\varepsilon>0$. Since $\proj$ is compact, there exists $\delta >0$ so that $\mid \varphi (x)-\varphi(y)\mid <\varepsilon$ whenever $d(x,y)<\delta$. Thus, taking $n\gg 0$ so that $q_n> \frac{1}{\delta}$ we get that $\mid \varphi (x)-\varphi(\frac{j}{q_n})\mid <\varepsilon$ for every $x\in [\frac{j}{q_n}, \frac{j+1}{q_n})$ and $j=0,1,\ldots,q_n-1$. In particular,

$$ \abs{\frac{1}{\nu_n([\frac{j}{q_n},\frac{j+1}{q_n}))} \int_{\frac{j}{q_n}}^{\frac{j+1}{q_n}}\varphi d\nu_n-\varphi(\frac{j}{q_n})}<\varepsilon.$$ 
Now, observing that $\nu_n([\frac{j}{q_n},\frac{j+1}{q_n}))=\frac{1}{q_n}$ for every $j=0,1,\ldots , q_n-1$ once $\nu_n$ is $R_{\theta_n}$-invariant, summing the previous expression for $j$ from $0$ up to $q_n-1$ and dividing both sides by $q_n$ we get that 
 $$
\abs{\int_0^{1}\varphi d\nu_n- \frac{1}{q_n}\sum_{j=0}^{q_n-1}\varphi(\frac{j}{q_n})}<\varepsilon.
$$
On the other hand, since $\varphi$ is Riemann integrable,
$$
\lim_{n\to\infty}\frac{1}{q_n}\sum_{j=0}^{q_n-1}\varphi(\frac{j}{q_n})=\int \varphi d\text{Leb}
$$
which implies that $\nu_n\xrightarrow{n\to +\infty} \text{Leb}$ as claimed. So, restricting to a subsequence, if necessary, we may assume that $\nu_n\xrightarrow{n\to +\infty} \text{Leb}$.

We now analyse the accumulation points of $\eta_n={P^{-1}_n}_* \nu_n$. If $\{P^{-1}_n\}_n$ stay in a compact set of $\SL$ then, taking a subsequence if necessary, we may assume that there exists $P\in \SL$ so that $P^{-1}_n\to P$. In particular, $\lim_{n\to \infty}\eta_n=P_*\text{Leb}$ which contradicts our assumption since $P_*{\text{Leb}}$ is non-atomic. If $\norm{P^{-1}_n}\to \infty$ then we can work on the compactification of quasi-projective transformations (see \cite{LLE} or \cite[Section 6.1]{BoV04}). In particular, restricting to a subsequence, if necessary, we have that $P^{-1}_n\to Q$, where $Q$ is defined outside some kernel (a one dimensional subspace) and the image $Im (Q)\subset \proj$ of $Q$ is a one dimensional subspace. Thus, as the kernel has zero Lebesgue measure we can apply \cite[Lemma~2.4]{AvV2} to conclude that
$$
\lim_{n\to\infty}{P^{-1}_n}_* \nu_n=Q_* \text{Leb}=\delta_{Im(Q)}
$$
which is a contradiction. Consequently, $L_n$ may be elliptic only for finitely many values of $n$.

To conclude the proof it remains to rule out the cases when $tr(L_n)=2$ and the matrix are non diagonalizable for infinitely many values of $n$. So, suppose $L_n$ is non diagonalizable and both of its eigenvalues are $1$ for every $n$. Then by the Jordan's normal decomposition we have
$$
L_n=P_n^{-1}\left(\begin{array}{cc}
1& 1\\
0&1
\end{array}\right) P_n
$$
for some $P_n\in GL(2,\mathbb{R})$. Consequently, the only invariant measure for $L_n$ is atomic and have only one atom contradicting the fact that $\eta_n\xrightarrow{n\to +\infty} \frac{1}{2}( \delta_p+\delta_q)$. Thus, $L_n$ can be parabolic and different from $\id$ only for finitely many values of $n$ concluding the proof of the proposition.
\end{proof}

\subsection{$PSL(2,\mathbb{R})$ cocycles} \label{sec: PSL2r}

Let us consider the \emph{projective special linear group} given by $PSL(2,\mathbb{R})=SL(2,\mathbb{R})/\{\pm Id\}$. That is, given $A,B\in SL(2,\mathbb{R})$ let $\sim$ be the equivalence relation given by $A\sim B$ if and only if $A=B$ or $A=-B$. Given $A\in SL(2,\mathbb{R})$, let $[[A]]=\{B\in SL(2,\mathbb{R}); B\sim A \}$ be the equivalence class of $A$ with respect to $\sim$. Then, $PSL(2,\mathbb{R})=\{[[A]]; A\in SL(2,\mathbb{R}) \}$. Observe that the norm $\norm{\cdot}$ on $\SL$ naturally induces a norm, which we are going to denote by the same symbol, on $PSL(2,\mathbb{R})$: given $A\in \SL$, $\norm{[[A]]}:=\norm{A}=\norm{-A}$.

Given $A:M\to SL(2,\mathbb{R})$ let us consider $\tilde{A}:M\to PSL(2,\mathbb{R})$ given by $\tilde{A}(x)=[[A(x)]]$. By Kingman's subadditive ergodic theorem \cite{Ki68} and the ergodicity of $\mu$ it follows that the limit
\begin{displaymath}
L(\tilde{A},\mu)=\lim_{n\to +\infty}\frac{1}{n}\log \|\tilde{A}^n(x)\| 
\end{displaymath}
exists and is constant for $\mu$-almost every $x\in M$. In particular, since $\|A^n(x)\|=\|\tilde{A}^n(x)\|$ for every $x\in M$ and $n\in \mathbb{N}$, we get that $\lambda^u(A,\mu)=L(\tilde{A},\mu)$. Another simple observation is that for every $v\in \Pone$, $[A(x)v]=[\tilde{A}(x)v]$ and, consequently, the action induced by $A$ on $\Pone$ coincide with the action of $\tilde{A}$ on $\Pone$. Moreover, $H^{\tilde{A}}_{\gamma} =[[H^{A}_{\gamma}]] \in PSL(2,\mathbb{R})$ is well defined and have similar properties with respect to $\tilde{A}$ as those of $H^{A}_{\gamma}$ with respect to $A$ described in Section \ref{sec: tivial holonomies}. In particular, a similar conclusion to that of Section \ref{sec: tivial holonomies} holds for $\tilde{A}$ whenever $H^{\tilde{A}}_{\gamma}=[[\id]]$ for every $(3K,L)$-loop $\gamma$: we can perform a change of coordinates that makes the cocycle $(\tilde{A},f)$ constant without changing $L(\tilde{A},\mu)$. Consequently, denoting this new cocycle by $\hat{\tilde{A}}$, it follows that $L(\tilde{A},\mu)$ is equal to logarithm of the norm of the greatest eigenvalue of any representative of $\hat{\tilde{A}}$.

Furthermore, the results of Section \ref{sec: SL2r} also have a counterpart for $PSL(2,\mathbb{R})$ cocycles. In order to state it, recall that a sequence $\{\tilde{L}_n\}_n$ in $PSL(2,\mathbb{R})$ is said to converge to $\tilde{L}\in PSL(2,\mathbb{R})$ if there are representatives $L$ and $L_n$ in $\SL$ of $\tilde{L}$ and $\tilde{L}_n$, respectively, so that the sequence $\{L_n\}_n$ converges to $L$ in $\SL$.

\begin{proposition}\label{l.hyperbolic.matrix2}
For each $n\in \mathbb{N}$, let $\tilde{L}_n \in PSL(2,\mathbb{R})$ be so that $\tilde{L}_n\xrightarrow{n\to +\infty} [[\id]]$ and let $\eta_n$ be an $\tilde{L}_n$-invariant measure on $\proj$ so that $\eta_n\xrightarrow{n\to +\infty} \frac{1}{2}( \delta_p+\delta_q)$ for some $p,q\in \proj$ with $p\neq q$. Then for every $n$ sufficiently large either $\tilde{L}_n$ is hyperbolic or $\tilde{L}_n=[[\id]]$.
\end{proposition}
This result follows easily from Proposition \ref{l.hyperbolic.matrix}: for every $\tilde{L}_n \in PSL(2,\mathbb{R})$ we can take a representative of $\tilde{L}_n$ in $\SL$ with positive trace and apply the aforementioned result to these representatives.

\section{Proof of the main result}
Let $f:M\to M$, $A:M\to SL(2,\mathbb{R})$ and $\mu $ be given as in Theorem \ref{teoA} and suppose there exists a sequence $\{A_k\}_{k\in \mathbb{N}}$ in $H^\alpha(M)$ with $\lambda^u(A_k,\mu)=\lambda^s(A_k,\mu)=0$ for every $k\in \mathbb{N}$ and such that $A_k\xrightarrow{k\to +\infty} A$.

For each $k\in \mathbb{N}$, let $m_k$ be an ergodic $F_{A_k}$-invariant probability measure on $M\times \proj$ projecting on $\mu$ where $F_{A_k}$ is defined similarly to $F_A$. Passing to a subsequence if necessary, we may assume that the sequence $\{m_k\}_{k}$ converges in the weak$^*$ topology to some measure $m$ which is, as one can easily check, $F_A$-invariant and projects on $\mu$. In order to prove Theorem \ref{teoA} we are going to analyse these families of measures and its respective disintegrations.

\subsection{Continuity and convergence of conditional measures}

It follows from Remark \ref{r.unifcont} and \cite[Theorem C]{ASV13} and its proof that 

\begin{corollary}\label{cor:equicont}
For every $k$ sufficiently large there exists an $su$-invariant disintegration $\{m^k_x:x\in M\}$ of $m_k$ with respect to the partition $\{\{x\}\times \proj:x\in M\}$ and $\mu$ such that 
$$\{M \ni x\to m^k_{x}\}_{k\gg 0} \text{ is equi-continuous.}$$
\end{corollary}

As an application of this corollary we get that

\begin{proposition}\label{p.convergence}
The measure $m$ is $su$-invariant and admits a continuous disintegration $\{m_x\}_{x\in M}$ with respect to $\{\{x\}\times \proj\}_{x\in M}$ and $\mu$ so that $m^k_x$ converges uniformly on $M$ to $m_x$.
\end{proposition}

In order to prove the previous proposition we need the following auxiliary result.

\begin{lemma}\label{l.prodconvergence}
Let $X$ and $Y$ be compact metric spaces, $\mu$ a Borel probability measure on $X$ and $\{\nu_k\}_{k\in \mathbb{N}}$ be a sequence of probability measures on $X \times Y$ projecting on $\mu$ and converging in the weak$^*$ topology to some measure $\nu$. Then for every measurable function $\rho:X\to \real$ and every continuous function $\varphi:Y\to \real$, 
$$
\lim_{k\to\infty} \int \rho\times \varphi d\nu_k=\int \rho\times\varphi d\nu.$$
\end{lemma}

\begin{proof}
Given $\varepsilon>0$ let $\hrho :X\to \real$ be a continuous function so that $\int_X\abs{\hrho-\rho}d\mu<\frac{\varepsilon}{2\sup{\varphi}}$. Take $k_0\in \mathbb{N}$ such that for every $k>k_0$, 
$$
\abs{\int \hrho\times \varphi d\nu_k-\int \hrho\times \varphi d\nu}<\frac{\varepsilon}{2}.
$$
Then, for $k>k_0$,
$$
\abs{\int \rho\times \varphi d\nu_k-\int \rho\times \varphi d\nu}<\sup{\varphi}\int_X\abs{\hrho-\rho}d\mu+\abs{\int \hrho\times \varphi d\nu_k-\int \hrho\times \varphi d\nu}<\varepsilon.
$$
\end{proof}

\begin{proof}[Proof of Proposition \ref{p.convergence}]
For each $k\in \mathbb{N}$, let $\{m^k_x\}_{x\in M}$ be the disintegration of $m_k$ given by Corollary \ref{cor:equicont}. We start observing that for every continuous function $\varphi:\proj\to\real$, by Arezel\`a-Aslcoi's theorem (recall Corollary \ref{cor:equicont}), there exists a subsequence of $\{\int_{\proj} \varphi dm^k_x \}_k$ such that $\int_{\proj} \varphi dm^{k_j}_x\to I_x(\varphi)$ uniformly on $M$. Taking a dense subset $\{\varphi_j\}_{j\in \mathbb{N}}$ of the space $C^0(\proj)$ of continuous functions $\varphi:\proj \to \mathbb{R}$ and using a diagonal argument, passing to a subsequence if necessary, we can suppose that  $\int_{\proj} \varphi dm^k_x\to I_x(\varphi)$ for every $\varphi \in C^0(\proj)$. It is easy to see that $I_x$ defines a positive linear functional on $C^0(\proj)$. Consequently, by Riesz-Markov's theorem, for every $x\in M$ there exists a measure $\hm_x$ on $\proj$ such that $I_x(\varphi)=\int \varphi d\hm_x$. 

On the other hand, letting $\{m_x\}_{x\in M}$ be a disintegration of $m$ with respect to $\{\{x\}\times \proj\}_{x\in M}$ and $\mu$ and invoking Lemma~\ref{l.prodconvergence} it follows that for every continuous function $\varphi:\proj\to\real$ and any $\mu$-positive measure subset $D\subset M$,
$$
\int _D\int_{ \proj}\varphi dm^k_x d\mu = \int_{D\times \proj} \varphi dm_k\to \int_{D\times \proj}\varphi dm =\int _D \int_{ \proj}\varphi dm_x d\mu.
$$ 
Consequently, $m_x=\hm_x$ for $\mu$ almost every $x\in M$. Thus, extending $m_x=\hm_x$ for every $x\in M$ we get a continuous disintegration of $m$ such that $m^k_x\to m_x$ uniformly on $x\in M$. In particular, by Remark \ref{r.unifcont} and the $su$-invariance of $m_k$ for every $k$ it follows that $m$ is also $su$-invariant as claimed.
\end{proof}

From now on we work exclusively with the disintegrations $\{m^k_x\}_{x\in M}$ and $\{m_x\}_{x\in M}$ of $m_k$ and $m$, respectively, given by Corollary \ref{cor:equicont} and the previous proposition.

Recall we are assuming $\lambda^u(A,\mu)>0>\lambda^s(A,\mu)$. Thus, letting $\mathbb{R}^2= E^{u,A}_{x}\oplus E^{s,A}_{x}$ be the Oseledets decomposition associated to $A$ at the point $ x\in M$, it follows from Proposition 3.1 of \cite{BP15} that for any $F_A$-invariant measure $m$, its conditional measures are of the form $m_x=a\delta_{E^{u,A}_x}+b\delta_{E^{s,A}_x}$ for some $a,b\in [0,1]$ such that $a+b=1$ where here and in what follows we abuse notation and identify a $1$-dimensional linear space $E$ with its class $[E]$ in $\Pone$. 

\begin{lemma} \label{lem: su-inv of Osel}
There exist continuous and $su$-invariant functions which coincide with $ x\to E^{s,A}_{x}, E^{u,A}_{x}$ for $\mu$-almost every point. By $su$-invariance we mean that for every (admissible) choice of $x,y,z\in M$, $H^{s,A}_{xy} E^{*}_x=E^*_y$ and $H^{u,A}_{xz} E^{*}_x=E^*_z$ for $*\in \{s,u\}$.
\end{lemma}
From now on we think of $ E^{s,A}_{x}$ and $E^{u,A}_{x}$ as continuous functions defined for every $x\in M$.
\begin{proof}
Recall $m_k$ is a $F_{A_k}$-invariant measure such that $m_k\to m$. Since $\lambda ^u(A_k,\mu)=0$ for every $k\in \mathbb{N}$ we get that $\int\Phi_{A_k}dm_k=0$ where $\Phi_{A_k}:M\times \proj \to \mathbb{R}$ is given by $\Phi_{A_k}(x,v)=\log \frac{\parallel A_k(x)v \parallel}{\parallel v\parallel}$. On the other hand,
$$
\int\Phi_{A_k}dm_k\to  \int\Phi_{A}dm.
$$
Thus, $ \int\Phi_{A}dm=0$ which implies that the numbers $a$ and $b$ given above are strictly larger than zero. Now, by Proposition~\ref{p.convergence} we know that $\{m_x\}_x$ is $su$-invariant. Consequently, since $E^{u,A}_x$ is $u$-invariant and $E^{s,A}_x$ is $s$-invariant, it follows  
$\delta_{E^{u,A}_x}=\frac{1}{a}(m_x-b\delta_{E^{s,A}_x})$ is also $s$-invariant. Analogously, $E^{s,A}_x$ is $u$-invariant. In particular, $E^{u,A}_x$ and $E^{s,A}_x$ are $su$-invariant. Continuity follows easily (see \cite[Theorem~D]{ASV13}).
\end{proof}

\subsection{Excluding the atomic case with a bounded number of atoms} \label{sec: atomic} 

In this subsection we prove that $m^k_{x_k}$ can not have a bounded number of atoms (with bound independent of $k$) for infinitely many values of $k\in \mathbb{N}$ and any $x_k\in M$. In order to do so, we need the following lemma.

\begin{lemma}\label{l.atomic}
If $m^k_y$ has an atom for some $y\in M$, then there exists $j=j(k)\in \mathbb{N}$ such that for every $x\in M$, there exist $v^1_x,\dots v^j_x\in \proj$ so that $m^k_x=\frac{1}{j}\sum_{i=1}^j\delta_{v^i_x}$.
\end{lemma}
\begin{proof}
Let $v_y\in \proj$ be such that $m^k_y(v_y)=\beta>0$ and for every $x\in M$, let $\gamma_x$ be an $su$-path joining $y$ and $x$. By the $su$-invariance of the disintegration $\{m^k_x\}_k$ it follows that $m^k_{x}(H^{A_k}_{\gamma_x} v_y)=\beta$ for every $x\in M$. Thus, considering $L=\lbrace (x,v_x)\in M\times \proj; \; m^k_x(v_x)=\beta\rbrace$ we get that $m_k(L)=\int m^k_x(L\cap \{x\}\times \proj)d\mu\geq\beta>0$. Consequently, since $L$ is $F_{A_k}$-invariant and $m_k$ is ergodic it follows that $m_k(L)=1$. In particular, $m^k_x(L\cap \{x\}\times \proj)=1$ for $\mu$-almost every $x\in M$ which implies that $m^k_x=\frac{1}{j}\sum_{i=1}^j\delta_{v^i_x} $, where $\frac{1}{j}=\beta$ (in particular, $j$ does not depend on $x$). Finally, to prove that this claim holds true \emph{for every} $x\in M$, we just take some $su$-path from a point in the total measure set and $x$ and use the $su$-invariance.
\end{proof}

The proof is going to be by contradiction. So, passing to a subsequence and using the previous lemma suppose $m^k_x$ has $j(k)$ atoms and that the sequence $\{j(k)\}_k$ is bounded. Restricting again to a subsequence, if necessary, we may assume that $j(k)$ is constant equal to some $j\in \mathbb{N}$. In particular, since $m_x=\frac{1}{2}\delta_{E^{s,A}_x}+\frac{1}{2}\delta_{E^{u,A}_x}$, for $k$ sufficiently large $m^k_x$ has an even number of atoms. Thus, writing $m^k_x=\frac{1}{j}\sum_{i=1}^j\delta_{v^i_k(x)}$ and reordering if necessary we may suppose that $v^i_k(x)\to E^{u,A}_x$ for $i\leq \frac{j}{2}$ and $v^\ell_k(x)\to E^{s,A}_x$ for $\ell>\frac{j}{2}$. Moreover, such convergence is uniform. Observe now that for each $k\in \mathbb{N}$ there exists some $x_k\in M$ such that $A_k(x_k)v^{i_k}_k(x_k)=v^{j_k}_k(f(x_k))$ for some $i_k\leq \frac{j}{2}$ and $j_k> \frac{j}{2}$, otherwise the set $L=\cup_{x\in M} \lbrace x\rbrace\times \lbrace v^1_k(x),\dots v^{\frac{j}{2}}_k(x)\rbrace$ would be $F_{A_k}$-invariant with measure
$$
m_k(L)=\int m^k_x(\lbrace v^1_k(x),\dots v^{\frac{j}{2}}_k(x)\rbrace)d\mu=\frac{1}{2},
$$
contradicting the ergodicity. Thus, restricting to a subsequence, if necessary, we may assume without loss of generality that $v^{i_k}_k(x_k)=v^{1}_k(x_k)$ and $v^{j_k}_k(x_k)=v^{j}_k(x_k)$ for every $k\in \mathbb{N}$ and that $x_k\to x$. In particular, 
$$
A(x)E^{u,A}_x=\lim_{k\to \infty} A_k(x_k)v^1_k(x_k)=\lim_{k\to\infty}v^j_k(f(x_k))=E^{s,A}_{f(x)} ,
$$
a contradiction. Summarizing, we can not have a subsequence $\{k_i\}_i$ so that the sequence $\{j(k_i)\}_{i}$ is bounded where $j(k)$ stands for the number of atoms of $m^k_x$ (which is independent of $x\in M$).

\subsection{Conclusion of the proof} \label{sec: concluison}
Given $x\in M$ let $\gamma$ be a non-trivial $su$-loop at $x$. In particular, from Lemma \ref{lem: su-inv of Osel} it follows that ${H^A_{\gamma}}E^{*,A}_x=E^{*,A}_x$ for $*\in \{s,u\}$. Consequently, either $H^A_{\gamma}$ is hyperbolic or $H^A_{\gamma}=\pm \id$. If $H^A_{\gamma}$ is hyperbolic then, since $H^{A_k}_{\gamma}\xrightarrow{k\to +\infty}H^A_{\gamma}$, it follows that $H^{A_k}_{\gamma}$ is also hyperbolic for every $k\gg0$. Thus, since ${H^{A_k}_{\gamma}}_*m^k_x=m^k_x$, it follows that $m^k_x$ is atomic and has at most two atoms for every $k\gg 0$ but from Section \ref{sec: atomic} we know this is not possible. So, we get that $H^{A}_{\gamma}=\pm\id$ \emph{for every} $su$-loop at $x$ and \emph{every} $x\in M$ and therefore $H^{\tilde{A}}_{\gamma}=[[\id]]$ \emph{for every} $su$-loop at $x$ and \emph{every} $x\in M$. Consequently, from Proposition \ref{l.hyperbolic.matrix2} we get that either there exists a non-trivial $su$-loop $\gamma$ at some point $x\in M$ and a sequence $\{k_j\}_j$ going to infinite as $j\to+\infty$ so that $H^{\tilde{A}_{k_j}}_{\gamma}$ is hyperbolic for every $j$ and thus $H^{A_{k_j}}_{\gamma}$ is also hyperbolic for every $j$ or $H^{\tilde{A}_k}_{\gamma}=[[\id]]$ for every $su$-loop $\gamma$ and every $k> k_\gamma$ for some $k_\gamma \in \mathbb{N}$. Arguing as we did above we conclude that the first case can not happen. So, all we have to analyse is the case when $H^{\tilde{A}_k}_{\gamma}=[[\id]]$ for every $su$-loop $\gamma$ and every $k> k_\gamma$ for some $k_\gamma \in \mathbb{N}$. 

If there exists $k_0\in \mathbb{N}$ so that $k_\gamma\leq k_0$ for every $su$-loop $\gamma$ then making the change of coordinates given in Section \ref{sec: tivial holonomies} for every $k>k_0$ (recall Section \ref{sec: PSL2r}) we get the that $L(\tilde{A}_k,\mu)$ is equal to the logarithm of the norm of the greatest eigenvalue of any representative of $\hat{\tilde{A}}_k(x)$, where $\hat{\tilde{A}}_k(x)$ is a constant element of $PSL(2,\mathbb{R})$, and $\hat{\tilde{A}}_k(x)\to \hat{\tilde{A}}(x)$. In particular, 
$$\lambda^u(A_k,\mu)=L(\tilde{A}_k,\mu)\xrightarrow{k\to+\infty}L(\tilde{A},\mu)=\lambda^u(A,\mu)$$
which is a contradiction. Now, recalling that in order to perform the change of coordinates in Section \ref{sec: tivial holonomies} it is enough to assume that $H^{\tilde{A}_k}_{\gamma}=[[\id]]$ for every $(K',L')$-loop $\gamma$ for some $K',L'>0$, to conclude the proof of Theorem \ref{teoA}, in view of the previous argumment, we only have to show that we can not have $k_\gamma$ arbitrarly large for $(K',L')$-loops. 

Let $k_\gamma$ be minimum for its defining property, that is, $H^{\tilde{A}_k}_{\gamma}=[[\id]]$ for every $k>k_\gamma$ and $H^{\tilde{A}_{k_\gamma}}_{\gamma}\neq[[\id]]$ and suppose that for each $j\in \mathbb{N}$ there exist $x_j\in M$ and a $(K',L')$-loop $\gamma_j$ at $x_j$ so that $k_{\gamma_j}\xrightarrow{j\to +\infty} +\infty$.  Passing to a subsequence we may assume $x_j\xrightarrow{j\to +\infty}x$ and $\gamma_j\xrightarrow{j\to +\infty}\gamma$ where $\gamma$ is an $su$-loop at $x$. This can be done because each $\gamma_j$ has at most $K'$ legs and each of them with length at most $L'$. In particular, if $\gamma_j$ is defined by the sequence $x_j=z^j_0,z^j_1, \ldots, z^j_{n_j}=x_j$ then $n_j\leq K'$ for every $j$. Thus, passing to a subsequence we may assume $n_j=n\leq K'$ for every $j\in \mathbb{N}$ and $z^j_i\xrightarrow{j\to +\infty} x_i$ for every $i=1,\ldots,n$ and consequently $\gamma$ is the $su$-loop defined by the sequence $x=x_0,x_1, \ldots, x_{n}=x$. Now, since $H^{\tilde{A}}_{\gamma}=[[\id]]$,  $H^{\tilde{A}_{k_{\gamma_j}}}_{\gamma_j}\xrightarrow{j\to +\infty} H^{\tilde{A}}_{\gamma}$ and $H^{\tilde{A}_{k_{\gamma_j}}}_{\gamma_j}\neq[[\id]]$ it follows from Proposition \ref{l.hyperbolic.matrix2} (recall Proposition \ref{p.convergence}) that $H^{\tilde{A}_{k_{\gamma_j}}}_{\gamma_j}$ is hyperbolic for every $j\gg 0$ and thus $H^{A_{k_{\gamma_j}}}_{\gamma_j}$ is also hyperbolic for every $j\gg 0$. Consequently, $m^{k_{\gamma_j}}_x$ is atomic and has at most two atoms for every $x\in M$ and every $j\in \mathbb{N}$ which again from Section \ref{sec: atomic} we know is not possible concluding the proof of Theorem \ref{teoA}.

\begin{remark}
We observe that Theorem \ref{teoA} can also be proved using the technics of couplings and energy developed in \cite{BBB}. Maybe those ideas can be useful in proving Conjecture \ref{conjecture}. We chose to present the previous proof because it is shorter and also different. It is also worth noticing that a similar result was obtained by Liang, Marin and Yang \cite[Theorem 6.1]{LMY} for the derivative cocycle under the additional assumption that $f$ has a pinching hyperbolic periodic point. In our context, such a hypothesis would immediately imply that all the conditional measures $m^k_x$ are atomic with at most two atoms for every $k\gg0$. In particular, Theorem \ref{teoA} would follow from the results of Section \ref{sec: atomic}.
\end{remark}

\section{Examples}

At this section we present two examples of fiber-bunched cocycles with nonvanishing Lyapunov exponents over a partially hyperbolic map which are accumulated by cocycles with zero Lyapunov exponents.

\subsection{Proof of Theorem \ref{teoB}}

Let $\omega$ be an irrational number of bounded type and $f_0:\mathcal{S}^1\to \mathcal{S}^1$ be given by $f_0(t)=t+2\pi \omega$ where $\mathcal{S}^1$ is the unit circle. Recently, Wang and You \cite[Theorem 1]{WaY15} constructed examples of cocycles
$A\in C^r(\mathcal{S}^1, \mbox{SL}(2,\mathbb{R}))$ over $f_0$, for any $r=0,1,\ldots ,\infty$ fixed, with arbitrarily large Lyapunov exponents which are approximated in the $C^r$-topology by cocycles with zero Lyapunov exponents. Let $A_0:\mathcal{S}^1\to \mbox{SL}(2,\mathbb{R})$ be such a cocycle and $\{A_k\}_k$ be a sequence in $C^r(\mathcal{S}^1, \mbox{SL}(2,\mathbb{R}))$ converging to $A$ so that $\lambda ^u(A_{k}, \nu)=0$ for every $k\in \mathbb{N}$ where $\nu$ denotes the Lebesgue measure on $\mathcal{S}^1$. Now, given $f_1:N\to N$, a volume-preserving Anosov diffeomorphism of a compact manifold $N$, let us consider the map $f:M:=\mathcal{S}^1\times N \to M$ given by $f(t,x)=(f_0(t),f_1(x))$ and let $\hat{A}:M\to \SL$ be given by $\hat{A}(t,x)=A_0(t)$. Thus, defining $\hat{A}_k(t,x)=A_k(t)$ and denoting by $\mu$ the Lebesgue measure on $M$ we have that $\lim _{k\to +\infty}\hat{A}_{k}=\hat{A}$, $\lambda ^u(\hat{A}_{k}, \mu)= \lambda ^u(A_{k}, \nu)=0$ for every $k \in \mathbb{N}$ and $\lambda ^u(\hat{A}, \mu)= \lambda ^u(A_0, \nu)>0$. Consequently, since $f$ is a volume-preserving partially hyperbolic and center-bunched diffeomorphism and $f_1$ may be chosen so that $(\hat{A},f)$ is fiber-bunched, we complete the proof of Theorem \ref{teoB}.

\subsection{Random product cocycles}
We now present another construction showing that given any real number $\lambda>0$, we have a fiber-bunched cocycle $A$ over a partially hyperbolic and center-bunched map $f$ so that $\lambda ^u(A, \mu)=\lambda$ which can be approximated by cocycles with zero Lyapunov exponents. We start with a general construction.

Let $\Sigma=\{1,\ldots ,k\}^{\mathbb{Z}}$ be the space of bilateral sequences with $k$ symbols and $\sigma:\Sigma\to \Sigma $ be the left shift map. Given maps $f_j:K\to K$ and $A_j:K\to \SL$ for $j=1,\ldots,k$ where $K$ is a compact manifold, let us consider 
$f:\Sigma\times K \to \SL$ and $A:\Sigma\times K \to \SL$ given, respectively, by
\begin{displaymath}
f(x,t)=(\sigma(x),f_{x_0}(t))
\end{displaymath}
and
\begin{displaymath}
A(x,t)=A_{x_0}(t).		
\end{displaymath}		

The \textit{random product of the cocycles} $\{(A_j,f_j)\}_{j=1}^k$ is then defined as the cocycle over $f$ which is generated by $A$. Observe that this definition generalizes the notion of random products of matrices explaining our terminology. Indeed, taking $K$ as being a single point we recover the aforementioned notion.

Differently from the case of random products of matrices where one have continuity of Lyapunov exponents (see \cite{BBB},\cite{BockerV}, \cite{LLE}), in the setting of random products of cocycles Lyapunov exponents can be very `wild'. This is what we exploit to construct our next example.

Let $f_0:\mathcal{S}^1 \to \mathcal{S}^1$ and $\nu $ be as in the previous example and let $A_0\in C^r(\mathcal{S}^1, \mbox{SL}(2,\mathbb{R}))$ be given by \cite[Theorem 1]{WaY15} so that $\lambda^u(A_0, \nu)>\lambda$. Taking $f_1:\mathcal{S}^1\to \mathcal{S}^1$ to be $f_1(t)=t$ and $A_1:\mathcal{S}^1\to \mbox{SL}(2,\mathbb{R})$ given by $A_1(t)=\id$, let $(A,f)$ be the random product of the cocycles $(A_0,f_0)$ and $(A_1,f_1)$ as defined above. Thus, letting $\eta$ be the Bernoulli measure on $\Sigma$ defined by the probability vector $(p_0,p_1)$ where $p_0$ is so that $p_0\lambda^u(A_0, \nu)=\lambda$ and considering $\mu=\eta \times \nu$, the cocycle generated by $A$ over $f$ has positive Lyapunov exponents and is accumulated by cocycles with zero Lyapunov exponents. Indeed, let $\{A_{0,k}\}_k$ be a sequence in $C^r(\mathcal{S}^1, \mbox{SL}(2,\mathbb{R}))$ converging to $A_0$ for which the cocycle $(A_{0,k},f_0)$ satisfies $\lambda ^u(A_{0,k}, \nu)=0$ for every $k\in \mathbb{N}$ whose existence is guaranteed by our choice of $A_0$ and \cite[Theorem 1]{WaY15}, $\{A_{1,k}\}_k$ be the sequence such that $A_{1,k}=\id$ for every $k\in \mathbb{N}$ and $(A_k, f)$ be the random product of $(A_{0,k},f_0)$ and $(A_{1,k},f_1)$. It is easily to see that $A_k \xrightarrow{k\to \infty}A$. Now, for $\mu$-almost every $(x,t)\in \Sigma\times \mathcal{S}^1$,
\begin{displaymath}
\lambda^u(A_k, \mu, x,t)=\lim_{n\to \infty}\frac{1}{n}\log\norm{A^n_k(x,t)}.
\end{displaymath}
Thus, observing that $A^n_k(x,t)=A^{\tau_n(x)}_{0,k}(t)$ where 
$$
\tau_n(x)=\#\left\{1\leq j \leq n;\; \sigma^j(x)_{0}=0\right\},
$$
it follows that
\begin{displaymath}
\lambda^u(A_k,\mu,x,t)=\lim_{n\to \infty}\frac{\tau_n(x)}{n}\frac{1}{\tau_n(x)}\log\norm{A^{\tau_n(x)}_{0,k}(t)}=p_0\lambda^u(A_{0,k},\nu).
\end{displaymath}
In partitular, $\lambda^u(A_k,\mu,x,t)$ is constant equal to $\lambda^u(A_k,\mu)$ for $\mu$-almost every $(x,t)\in \Sigma\times \mathcal{S}^1$. Analogously, $\lambda ^u(A,\mu)=p_0 \lambda ^u(A_{0}, \nu)$. Consequently,
$$\lambda^u(A_k,\mu)=0 \text{ for every } k\in \mathbb{N} \text{ and } \lambda _u(A,\mu)=\lambda>0$$
as claimed. Observe that despite the fact of not being smooth, the map $f$ is partially hyperbolic in the sense of the expansion and contraction properties when $\Sigma$ is endowed with the usual metric. Moreover, it is center-bunched and the cocycle $A$ is fiber-bunched.

\section*{Acknowledgements} We thank to Karina Marin for many helpful
comments and suggestions on this work and also for pointing out a gap in a previous version of the concluding argument. The first author was partially supported by a CAPES-Brazil postdoctoral fellowship under Grant No. 88881.120218/2016-01 at the University of Chicago. The second author was partially supported by Universit\'e Paris 13. The third author was partially supported by Universidad de Costa Rica and CNPq-Brazil.

\end{document}